\documentclass[10pt]{amsart}

\usepackage{amsmath, amsfonts, amsthm, amssymb, graphicx, fullpage, enumerate}

\newtheorem{theorem}{Theorem}      
\newtheorem{lemma}{Lemma}
\newtheorem{corollary}{Corollary}

\newtheorem*{main theorem}{Main Theorem}   
\newtheorem*{thmA}{Theorem A}  
\newtheorem*{thmB}{Theorem B}    
\newtheorem*{thmC}{Theorem C}

\theoremstyle{remark}  
 
\theoremstyle{definition}  
\newtheorem{definition}{Definition}  
     
\def\E{\mathbb{E}}
         
\def\N{\mathbb{N}}     
         
\def\R{\mathbb{R}}     
\def\Z{\mathbb{Z}}

\def\norm#1{\|#1\|}

\begin{document}
\title{A Quantitative Result on Diophantine Approximation for Intersective Polynomials }
\author{Neil Lyall \quad\quad\quad Alex Rice }

\address{Department of Mathematics, The University of Georgia, Athens, GA 30602, USA}
\email{lyall@math.uga.edu}
\address{Department of Mathematics, Bucknell University, Lewisburg, PA 17837, USA}
\email{alex.rice@bucknell.edu} 
\begin{abstract} In this short note, we closely follow the approach of Green and Tao \cite{GT} to extend the best known bound for recurrence modulo 1 from squares to the largest possible class of polynomials. The paper concludes with a brief discussion of  a consequence of this result for polynomials structures in sumsets and limitations of the method.
\end{abstract}
\maketitle
\setlength{\parskip}{5pt} 
\section{Introduction}
We begin by recalling the well-known Kronecker approximation theorem:
\begin{thmA}[Kronecker Approximation Theorem]
Given $\alpha_1,\dots,\alpha_d\in \R$ and $N\in\N$, there exists an integer $1\leq n\leq N$ such that $$\|n\alpha_j\|\ll N^{-1/d}  \text{ for all } 1\leq j \leq d.$$
\end{thmA}

\emph{Remark on Notation:} In Theorem A above, and in the rest of this paper, we use the standard notations $\norm{\alpha}$ to denote, for a 
given $\alpha\in\R$, the distance from  $\alpha$ to the nearest integer and the Vinogradov symbol $\ll$ to denote ``less than a constant times''.

Kronecker's theorem is of course an almost immediate consequence of the pigeonhole principle: one simply partitions the torus $(\R/\Z)^d$ into $N$ ``boxes" of side length at most $2N^{-1/d}$ and   considers the orbit of $(n\alpha_1,\dots, n\alpha_d)$. 
In \cite{GT}, Green and Tao presented a proof of the following quadratic analogue of the above theorem, due to Schmidt \cite{Schmidt}.

\begin{thmB}[Simultaneous Quadratic Recurrence, Proposition A.2 in \cite{GT}]
Given $\alpha_1,\dots,\alpha_d\in \R$ and $N\in\N$, there exists an integer $1\leq n\leq N$ such that \[\|n^2\alpha_j\|\ll dN^{-c/d^2} \text{ \ for all \ }1\leq j \leq d.\] 
\end{thmB}

The argument presented by Green and Tao in \cite{GT} was later extended (in a straightforward manner) by the second author and Magyar in \cite{LM} to any system of polynomials without constant term.

\begin{thmC}[Simultaneous Polynomial Recurrence, consequence of Proposition B.2 in \cite{LM}] \label{sim}
Given any system of polynomials $h_1,\dots,h_d$ of degree at most $k$ with real coefficients and no constant term and $N\in\N$, there exists an integer $1\leq n\leq N$ such that \[\|h_j(n)\|\ll k^2d N^{-ck^{-C}/d^2} \text{ \ for all \ }1\leq j \leq d,\] where $C,c>0$ and the implied constant are absolute.\end{thmC}

Such a recurrence result does not hold for every polynomial. Specifically, if $h\in \Z[x]$ has no root modulo $q$ for some $q\in \N$, then $\norm{h(n)/q}\geq 1/q$ for all $n\in \Z$, a local obstruction which leads to the following definition.

\begin{definition} We say that $h\in \Z[x]$ is \textit{intersective} if for every $q\in \N$, there exists $r\in \Z$ with $q \mid h(r)$. Equivalently, $h$ is intersective if it has a root in the $p$-adic integers for every prime $p$.
\end{definition}

Intersective polynomials include all polynomials with an integer root, but also include certain polynomials without rational roots, such as $(x^3-19)(x^2+x+1)$. 

\section{Recurrence for Intersective Polynomials}

The purpose of this note is to extend the argument of Green and Tao \cite{GT} to establish the following quantitative improvement of a result of L\^e and Spencer \cite{LS}. 

\begin{theorem}\label{main} Given $\alpha_1, \dots, \alpha_d \in \R$, an intersective polynomial $h\in \Z[x]$ of degree $k$, and $N\in \N$, there exists an integer $1\leq n \leq N$ with $h(n)\neq 0$ and $$\norm{h(n)\alpha_j} \ll dN^{-c^k/d^2} \text{ for all } 1\leq j \leq d,$$ where $c>0$ is absolute and the the implied constant depends only on $h$.
\end{theorem}

\noindent In \cite{LS}, the right hand side is replaced with $N^{-\theta}$ for some $\theta=\theta(k,d)>0$. Here we follow Green and Tao's \cite{GT} refinement of Schmidt's \cite{Schmidt} lattice method nearly verbatim, beginning with the following definitions. 

\begin{definition} Suppose that $\Lambda \subseteq \R^d$ is a full-rank lattice. For any $t>0$ and $x=(x_1,\dots,x_d)\in \R^d$, we define the \textit{theta function} $$\Theta_{\Lambda}(t,x) := \sum_{m\in \Lambda} e^{-\pi t|x-m|^2}.$$
Further, we define $$A_\Lambda:=\Theta_{\Lambda^*}(1,0)=\sum_{\xi \in \Lambda^*} e^{-\pi |\xi|^2}=\det(\Lambda)\sum_{m\in \Lambda}e^{-\pi|m|^2},$$ 
where $\Lambda^*=\{\xi\in \R^d \ : \xi \cdot m \in \Z \text{ for all } m\in \Lambda \}$ and the last equality follows from the Poisson summation formula. Finally, for a polynomial $h\in \Z[x]$,  $\alpha=(\alpha_1,\dots,\alpha_d)\in \R^d$, and $N>0$, we define 
$$F_{h,\Lambda,\alpha}(N):= \det(\Lambda)\E_{1\leq n \leq N} \Theta_{\Lambda}(1,h(n)\alpha).$$
\end{definition}

\noindent For the remainder of the discussion, we fix an intersective polynomial $h\in \Z[x]$ of degree $k$, and we let $K=2^{10k}$. We use $C$ and $c$ to denote sufficiently large and small absolute constants, respectively, and we allow any implied constants to depend on $h$. By definition $h$ has a root at every modulus, but we need to fix a particular root at each modulus in a consistent way, which we accomplish below.

\begin{definition}\label{auxdef} For each prime $p$, we fix $p$-adic integers $z_p$ with $h(z_p)=0$. By reducing and applying the Chinese Remainder Theorem, the choices of $z_p$ determine, for each natural number $q$, a unique integer $r_q \in (-q,0]$, which consequently satisfies $q \mid h(r_q)$. We define the function $\lambda$ on $\N$ by letting $\lambda(p)=p^m$ for each prime $p$, where $m$ is the multiplicity of $z_p$ as a root of $h$, and then extending it to be completely multiplicative. 

\noindent For each $q\in \N$, we define the \textit{auxiliary polynomial}, $h_q$, by 
\begin{equation*} h_q(x)=h(r_q + qx)/\lambda(q), \end{equation*}
noting that each auxiliary polynomial maintains integral coefficients.
\end{definition}

\noindent As in \cite{GT}, we make use of the following properties of $F$, only one of which needs to be tangibly modified due to the presence of a general intersective polynomial.

\begin{lemma} [Properties of $F_{h_q, \Lambda, \alpha}$] If $\Lambda \subseteq \R^d$, $\alpha \in \R^d$, and $q, N \in \N$, then

\begin{enumerate} 
\item[\textnormal{(i)}] \textnormal{(Contraction of $N$)} $F_{h_q, \Lambda, \alpha}(N) \gg cF_{h_q, \Lambda, \alpha}(cN)$ for any $c\in (10/N, 1)$.
\item[\textnormal{(ii)}] \textnormal{(Dilation of $\alpha$)} $F_{h_q, \Lambda, \alpha}(N) \gg \frac{1}{q'} F_{h_{qq'}, \Lambda, \lambda(q')\alpha}(N/q') $ for any $q'\leq N/10$.
\item[\textnormal{(iii)}] \textnormal{(Stability)} If $\tilde{\alpha}\in \R^d$ with $|\alpha-\tilde{\alpha}|<\epsilon/\displaystyle{\max_{1\leq n \leq N}|h_q(n)|}$ and $\epsilon\in (0,1)$, then $$F_{h_q, \Lambda, \alpha}(N) \gg F_{h_q, (1+\epsilon)\Lambda, (1+\epsilon)\tilde{\alpha}}(N).$$
\end{enumerate}
\end{lemma}

\begin{proof} Property (i) follows immediately from the definition of $F$ and the positivity of $\Theta$, and property (iii) is exactly as in Lemma A.5 in \cite{GT}. For property (ii), by positivity of $\Theta$, complete multiplicativity of $\lambda$, and the fact that $r_q \equiv r_{qq'} \text{ mod } qq'$, we have 
\begin{align*} F_{h_q, \Lambda, \alpha}(N)& =   \det(\Lambda)\E_{\substack{r_q+q\leq n \leq r_q+qN \\ n\equiv r_q \text{ mod }q }} (1,h(n)\alpha/\lambda(q))\\& \geq \det(\Lambda)\E_{\substack{r_q+q\leq n \leq r_q+qN \\ n\equiv r_{qq'} \text{ mod }qq' }} (1,h(n)\alpha/\lambda(q)) \\ & \gg \frac{1}{q'} \det(\Lambda) \E_{1\leq n \leq N/q'} \Theta_{\Lambda}\Big(1,\frac{h(r_{qq'}+qq'n)}{\lambda(qq')}\lambda(q')\alpha \Big) \\ & = \frac{1}{q'} F_{h_{qq'}, \Lambda, \lambda(q')\alpha}(N/q'),
\end{align*}
as required.
\end{proof}
 
\noindent The key to the argument is the following ``alternative lemma."\\

\begin{lemma} [Schmidt's Alternative] \label{alt} If $\Lambda\subseteq \R^{d}$ is a full-rank lattice, $\alpha\in \R^{d}$, and $q\leq N^{1/K}$, then one of the following holds:

\begin{enumerate}
\item[\textnormal{(i)}] $F_{h_q, \Lambda, \alpha} (N) \geq 1/2$\\
\item[\textnormal{(ii)}] There exists $q' \ll dA_{\Lambda}^{Ck}$ and a primitive $\xi \in \Lambda^* \setminus \{0\}$ such that 
\begin{equation*} |\xi| \ll \sqrt{d} + \sqrt{\log A_{\Lambda}}
\end{equation*}
and 
\begin{equation*} \norm{q'\xi \cdot \alpha} \ll A_{\Lambda}^{Ck}N^{-k}.
\end{equation*}
\end{enumerate}
\end{lemma}

\noindent The proof of Lemma \ref{alt} is identical to that of the corresponding lemma in \cite{GT}, once armed with the following result, which follows from Weyl's Inequality and observations of Lucier \cite{Lucier} on auxiliary polynomials. \\

\begin{lemma} \label{weyl} If $\delta \in (0,1)$, $q\leq N^{1/K}$, and $|\E_{1\leq n \leq N} \ e^{2\pi i h_q(n) \theta}|\geq \delta$, then there exists $q' \ll \delta^{-k}$ such that $\norm{q'\theta} \ll (\delta N)^{-k}$.
\end{lemma}

\noindent Additionally, a proof of Lemma \ref{weyl} is contained in Section 6.4 of \cite{thesis}. Precisely as in \cite{GT}, the alternative lemma gives the following inductive lower bound on $F$.\\
  

\begin{corollary}[Inductive lower bound on $F_{h,\Lambda, \alpha}$] \label{ind} If  $\Lambda\subseteq \R^{d}$ is a full-rank lattice, $\alpha\in \R^{d}$, $ N > (dA_{\Lambda})^{C_0k} $ for a suitably large absolute constant $C_0$, and $q< N^{1/K}$, then one of the following holds:
\begin{enumerate}
\item[\textnormal{(i)}] $F_{h_q, \Lambda, \alpha} (N) \geq 1/2$\\
\item[\textnormal{(ii)}] There exists $\alpha' \in \R^{d-1}$, a full-rank lattice $\Lambda'\subseteq \R^{d-1}$, $N' \gg (dA_{\Lambda})^{-Ck}N$, and $q' \ll (dA_{\Lambda})^{Ck}$ with

\begin{equation} \label{AL} A_{\Lambda'} \ll (\sqrt{d} + \sqrt{\log A_{\Lambda}})A_{\Lambda}
\end{equation} and
\begin{equation}\label{ind2} F_{h_q, \Lambda, \alpha} (N) \gg (dA_{\Lambda})^{-Ck} F_{h_{qq'}, \Lambda', \alpha'}(N'). 
\end{equation}

\end{enumerate} 

\end{corollary}

\

\noindent Finally, we use Corollary \ref{ind} to obtain a lower bound on $F_{h, \Lambda, \alpha}$ that is sufficient to prove Theorem \ref{main}.\\

\begin{corollary}\label{cor} If $\alpha \in \R^d$, $\Lambda \subseteq \R^d$ is a full-rank lattice with $\det(\Lambda)\geq 1$, and $N>(dA_{\Lambda})^{C_1kKd}$ for a suitably large absolute constant $C_1$, then $$F_{h, \Lambda, \alpha}(N) \gg (dA_{\Lambda})^{-Ckd}. $$

\end{corollary}

\begin{proof}
Setting $\alpha_0=\alpha$, $\Lambda_0=\Lambda$, and $N_0=N$, we repeatedly apply Corollary \ref{ind}, obtaining vectors $\alpha_j \in \R^{d-j}$, lattices $\Lambda_j \subseteq \R^{d-j}$, and integers $q_j, N_j$ for $j=0,1,\dots$. Assuming that $N_j>(dA_{\Lambda_j})^{C_0k}$ and $q_j\leq N_j^{1/K}$ throughout the iteration, which we will show to be the case shortly, we must either pass through case \textit{(i)} of Proposition \ref{ind} at some point, or the iteration continues all the way to dimension $0$. The worst bounds come from the latter scenario, and we note that $F_{h_{q_d},\Lambda_d, \alpha_d}(N_d)=1$. Using (\ref{AL}) and the crude inequality $\sqrt{d}+\sqrt{\log X} \ll dX^{1/d}$, we see that $A_{\Lambda_j}\ll A_{\Lambda_0}^C$ throughout the iteration. Since $N_{j+1}\geq (dA_{\Lambda_j})^{-Ck}N_j$ and $q_{j+1} \ll (dA_{\Lambda_j})^{Ck}q_j$, we see that $N_j>(dA_{\Lambda_j})^{C_0k}$ and $q_j \leq N_j^{1/K}$ throughout, provided $N\geq (dA_{\Lambda})^{C_1kKd}$ for suitably large $C_1$. From (\ref{ind2}), the result follows. 
\end{proof}

\subsection{Proof of Theorem \ref{main}} Fix real numbers $\alpha_1, \dots, \alpha_d \in \R$ and an intersective polynomial $h\in \Z[x]$ of degree $k$. Let $R$ be a quantity to be chosen later, and apply Corollary \ref{cor} with $\alpha=(R\alpha_1,\dots, R\alpha_d)$ and $\Lambda = R\Z^d$. By definition we have 
\begin{equation*} A_{\Lambda}=R^d\Big(\sum_{m\in R\Z}e^{-\pi m^2}\Big) \leq (CR)^d,
\end{equation*} so if $R\geq C_2d$ and $N>C_2R^{C_2kKd^2}$ for suitably large $C_2$, Corollary \ref{cor} implies 
\begin{equation*} F_{h, \Lambda, \alpha}(N) \gg R^{-Ckd^2}.
\end{equation*}
Since $\det(\Lambda)=R^d$, it follows from the definition of $F_{h,\Lambda, \alpha}$ that 
\begin{equation*}\E_{1\leq n \leq N} \sum_{m\in R\Z^d} e^{-\pi|h(n)\alpha-m|^2} \gg R^{-Ckd^2}
\end{equation*} 
The contribution from all $n$ with $h(n)=0$ is $\ll (CR)^d/N$, which is negligible if $N>C_2R^{C_2kKd^2}$. In this case we conclude that there exists $n\in \{1,\dots, N\}$ with $h(n)\neq 0$ and 
\begin{equation}\label{avg} \sum_{m \in R\Z^d}e^{-\pi|h(n)\alpha-m|^2} \gg R^{-Ckd^2} 
\end{equation}
Fixing such an $n$, if we had $|h(n)\alpha-m|>\sqrt{R}$ for all $m\in R\Z^d$, then we would have 
\begin{equation} \label{R1} e^{-\pi|h(n)\alpha-m|^2}\leq e^{-\pi R^2/2}e^{-\pi|h(n)\alpha -m|^2/2} 
\end{equation}
for all $m\in R\Z^d$. By the Poisson summation formula, we have the identity
\begin{equation} \label{psf} \sum_{m\in \Lambda} e^{-\pi t |h(n)\alpha-m|^2} = \frac{1}{t^{d/2}\det(\Lambda)}\sum_{\xi \in \Lambda^*} e^{-\pi |\xi|^2/t}e^{2\pi i \xi \cdot h(n)\alpha}.
\end{equation}
Applying (\ref{R1}) and (\ref{psf}), we conclude that 
\begin{equation*}\sum_{m \in R\Z^d}e^{-\pi|h(n)\alpha-m|^2} \leq e^{-\pi R^2/2} \frac{2^{d/2}}{\det(\Lambda)}\sum_{\xi\in \Lambda^*}e^{-2\pi|\xi|^2}e^{2\pi i \xi \cdot h(n)\alpha} \leq e^{-\pi R^2/2} 2^{d/2} \frac{A_{\Lambda}}{\det(\Lambda)}, 
\end{equation*}
which is $\ll e^{-\pi R^2/2}(CR)^d$, which contradicts (\ref{avg}) if $R>C_2d$. Therefore, under this assumption on $R$, it must be the case that there exists $m\in R\Z^d$ with $|h(n)\alpha-m|\leq \sqrt{R},$ which clearly implies that $\norm{h(n)\alpha_i} \leq 1/\sqrt{R}$ for all $1\leq j \leq d$. 

\noindent If $N \geq C_3d^{C_3kKd^2}$ for suitably large $C_3$, then the theorem follows by choosing $R=d^{-1}N^{c/d^2kK}$ for a sufficiently small absolute constant $c>0$. If instead $N<C_3d^{C_3kKd^2}$, then the theorem is trivial. \qed 

\section{Consequences and Limitations}

\subsection{Consequences for sumsets following Croot-Laba-Sisask}  Croot, Laba, and Sisask \cite{CLS} displayed, using machinery from \cite{CS} and \cite{Sanders}, that for sets $A,B\subseteq \Z$ of small doubling, there exists a low rank, large radius Bohr set $T$ with the property that a shift of any (not too large) subset of $T$ is contained in the sumset $A+B=\{a+b :a\in A, b\in B\}$. The theorems discussed in this paper imply the existence of particular polynomial configurations in Bohr sets, and hence can be incorporated with the techniques found in \cite{CLS} to establish corresponding sumset results. Specifically, by replacing the Kronecker Approximation Theorem with Theorem \ref{main} and C, respectively, in the proof of Theorem 1.4 in \cite{CLS}, one obtains the following results.

\begin{theorem} \label{ss1} Suppose $h\in \Z[x]$ is an intersective polynomial of degree $k$, and $A,B\in \Z$ with $$|A+B|\leq K_A|A|, K_B|B|,$$ then $A+B$ contains an arithmetic progression $$\{x+h(n)\ell : 1\leq \ell\leq L\}$$ with $x,\in \Z$, $n\in\N$, $h(n)\neq 0$ and $$L\gg\exp\Big(c^k\Big(\frac{\log|A+B|}{K_B^2(\log 2K_A)^6}\Big)^{1/3}-C\log(K_A\log|A|)\Big), $$ where $C,c>0$ are absolute constants, and the implied constant depends only on $h$.
 
\end{theorem}

\begin{theorem} \label{ss2}
Suppose $h_1,\dots,h_{m}\in \Z[x]$ with $h_i(0)=0$ and $\deg(h_i)\leq k$ for $1\leq i\leq \ell$, and $A,B\in \Z$ with $$|A+B|\leq K_A|A|, K_B|B|,$$ then $A+B$ contains a configuration of the form $$\{x+h_i(n)\ell: 1\leq i\leq m, \ 1\leq \ell\leq L\}$$ with $x \in \Z$, $n\in\N$, $h_i(n)\neq 0$ for $1\leq i \leq m$, and $$L\gg\exp\Big(ck^{-C}\Big(\frac{\log|A+B|}{m^2K_B^2(\log 2K_A)^6}\Big)^{1/3}-C\log(mkK_A\log|A|)\Big), $$ where $C,c>0$ and the implied constant are absolute.     
\end{theorem}

\noindent Noting that if $A,B\subseteq [1,N]$ with $|A|=\alpha N$ and $|B|=\beta N$, then one can take $K_A=2\alpha^{-1}$ and $K_B=2\beta^{-1}$, yielding special cases of Theorems \ref{ss1} and \ref{ss2} phrased in terms of densities. 

\subsection{Limitations toward simultaneous recurrence} Upon inspection of Theorems C and \ref{main}, and correspondingly Theorems \ref{ss1} and \ref{ss2}, the natural question arises of the possibility of common refinements. Specifically, if $\alpha_1, \dots, \alpha_d\in \R$ and $h_1,\cdots, h_m \in \Z[x]$ is a \textit{jointly intersective} collection of polynomials, meaning the polynomials share a common root at each modulus, can one simultaneously control $\norm{h_i(n)\alpha_j}$ for $1\leq i \leq m$ and $1\leq j \leq d$? In a qualitative sense, L\^e and Spencer \cite{LS} answered this question in the affirmative, but in this context obstructions arise to the application of the methods found in \cite{LM} to establish a bound such as that found in Theorem \ref{main}.

\noindent For example, suppose $h_1(x)=b_0+b_1x+b_2x^2$ and $h_2(x)=c_0+c_1x+c_3x^3$. This system of polynomials is a ``nice" system as defined in \cite{LS}, but to apply the methods of \cite{LM} it is necessary to firmly control Gauss sums of the form $$\sum_{n=1}^N e^{2\pi i (h_1(n)a_1+h_2(n)a_2)/q}=\sum_{n=1}^N e^{2\pi i \Big(b_0a_1+c_0a_2+(b_1a_1+c_1a_2)n+b_2a_1n^2+c_3a_2n^3\Big)/q}.$$
Control of this sum is lost if $b_1a_1+c_2a_2$, $b_2a_1$, $c_3a_2$, and $q$ all share a large common factor. While the argument allows us to control $(b_1, b_2)$, $(c_1, c_3)$, and $(a_1,a_2,q)$, this does not prohibit the aforementioned fatal scenario. While it is likely that an analog of Theorem C holds for a jointly intersective collection of polynomials, it appears that new insight is required.

\end{document}